\documentclass[11p,leqno]{amsart}
\usepackage{amsmath}
\usepackage{amsfonts}
\usepackage{amssymb}
\usepackage{graphicx}
\usepackage{color}

\usepackage{amsfonts,mathtools,mathrsfs}
\usepackage{amscd}
\usepackage{amsmath,amsthm}
\newtheorem{theorem}{Theorem}[section]\newtheorem{thm}[theorem]{Theorem}
\newtheorem*{theorem*}{Theorem}
\newtheorem{lemma}{Lemma}[section]

\newtheorem{definition}[theorem]{Definition}

\newtheorem{remark}[theorem]{Remark}

\newcommand{\RR}{\mathbb{R}}

\def\Ric{\text{Ric}}

\def\a{\alpha}
\def\ol{\overline}

\def\e{\epsilon}
\def\p{\partial}

\def\R{\Bbb R}

\def\vp{\varphi}

\def\Ric{\operatorname{Ric}}

\def\tr{\operatorname{tr}}

\numberwithin{equation}{section}

\begin{document}
\title[]{\bf Moduli of continuity for viscosity solutions on manifolds}

\author{Xiaolong Li}
\address{Department of Mathematics, University of California, San Diego, La Jolla, CA 92093, USA}
\email{xil117@ucsd.edu}
%
%
\author{Kui Wang}
%
%
\address{School of Mathematic Sciences, Soochow University, Suzhou, 215006, China}
\email{kuiwang@suda.edu.cn}
%

\maketitle

\begin{abstract}
We establish the estimates of modulus of continuity for viscosity solutions of nonlinear evolution equations on manifolds, extending previous work of B. Andrews and J. Clutterbuck for regular solutions on manifolds \cite{AC3} and the first author's recent work for viscosity solutions in Euclidean spaces \cite{me1}.
\end{abstract}

\section{Introduction}

In this paper, we study the moduli of continuity for viscosity solutions to nonlinear evolution equations on manifolds.
Let $(M, g)$ be a compact Riemannian manifold. Recall that given a continuous function $u:M \to \R$, the \emph{optimal modulus of continuity} $w$ of $u$ can be defined by
$$w(s)=\sup\left\{\frac{u(y)-u(x)}{2}: \text{\quad} d(x,y)=2s\right\},$$
where $d$ is the induced distance function on $(M, g)$.

We will mainly consider the following isotropic flow:
\begin{equation}\label{1.1}
u_t = \left[\alpha(|Du|,t) \frac{D_iu D_ju}{|Du|^2}+\beta(|Du|,t)\left(\delta_{ij}-\frac{D_iu D_ju}{|Du|^2}\right)\right]D_iD_ju +b(|Du|, t).
\end{equation}
%
We make the assumptions that equation \eqref{1.1} is nonsingular, i.e., the right hand side of \eqref{1.1}
is a continuous function on $\R_{+} \times \R^n \times S^{n \times n}$,
where $S^{n \times n}$ is the set of $n \times n$ symmetric matrices, and that $\alpha, \beta$ are nonnegative functions.

For domains in Euclidean spaces, B. Andrews and J. Clutterbuck \cite{AC2} proved that the modulus of continuity for a regular solution of \eqref{1.1} is a viscosity subsolution of the one-dimensional equation $\phi_t = \alpha(\phi',t) \phi''$. Recently this was shown by the first author \cite{me1} to be true for viscosity solutions as well.
On manifolds the estimates of modulus of continuity for regular solutions have been investigated by B. Andrews and J. Clutterbuck \cite{AC3}, B. Andrews and L. Ni\cite{AN} and L. Ni\cite{N1}. More precisely, if the Ricci curvature of the manifold has a lower bound: $\Ric_g \geq (n-1) \kappa g$, then the modulus of continuity of a regular solution of \eqref{1.1} satisfies
\begin{equation}\label{1d mfd}
w_t\le \alpha(w', t)w''+(n-1)\frac{c_{\kappa}'(s)}{c_{\kappa}(s)} \beta(w', t)w' 
\end{equation}
in the viscosity sense, where $c_{\kappa}(s)$ is defined by $c_{\kappa}'' + \kappa c_{\kappa} =0, c_{\kappa}(0)=1, c_{\kappa}'(0)=0$. The main goal of this paper is to show that various modulus of continuity estimates remain valid for viscosity solutions on manifolds as well. We would like to mention some important examples of such equations:
\begin{itemize}
\item[(i)]If we take $\a= 1, \beta =  1$ and $b=0$, then equation \eqref{1.1} reduces to the heat equation:
$$u_t =\Delta u;$$
\item[(ii)]If we take $\alpha = \frac{1}{1+|Du|^2}, \beta =1 $ and $b=0$, then equation \eqref{1.1} reduces to the graphical mean curvature flow:
$$u_t=\left(\delta_{ij}-\frac{D_iu D_ju}{1+|Du|^2}\right)D_iD_ju;$$
\item[(iii)]If we take $\a = (p-1)|Du|^{p-2}, \beta=|Du|^{p-2}$ and $b=0$, then equation \eqref{1.1} reduces to the $p$-Laplacian equation with $p> 2$:
$$\mbox{div} \left(|Du|^{p-2}Du\right)=0.$$
\end{itemize}

The paper is organized as follows: In Section 2, we recall the definitions of viscosity solutions on manifolds and state the parabolic maximum principle for semicontinuous functions on manifolds, which is the main technical tool we use in this paper. The main proof is given in Section 3. In Section 4, we prove height dependent gradient bounds, which is useful to derive gradient estimates for nonlinear equations.  A generalization of Section 3 to Bakry-Emery manifolds is done in Section 5. In Section 6, we treat Neuman and Dirichlet boundary value problems and establish the estimates of modulus of continuity.

\section{Preliminaries}

\subsection{Definition of Viscosity Solutions on manifolds}
\mbox{ }\\
Let $M$ be a Riemannian manifold. The following notations are useful:
$$\mbox{USC}(M\times (0,T))=\{u:M \to \R |\  u \mbox{  is upper semicontinuous  }\},$$
$$\mbox{LSC}(M\times (0,T))=\{u:M \to \R |\  u \mbox{  is lower semicontinuous  }\}.$$

We first introduce the notion of parabolic semijets on manifolds. We write $z=(x,t)$ and $z_0=(x_0, t_0)$.
\begin{definition}
For a function $u\in \mbox{USC}(M\times (0,T))$,
we define the parabolic second order superjet of $u$ at a point $z_0\in M\times (0,T)$ by
\begin{align*}
\mathcal{P}^{2,+} u (z_0) &:=\{(\vp_t(z_0), D\vp(z_0), D^2\vp(z_0)) :
 \vp \in C^{2,1}(M\times (0,T)),  \\
  & \mbox{  such that  } u- \vp \mbox{  attains a local maximum at } z_0\}.
\end{align*}
For $u\in \mbox{LSC}(M\times (0,T))$, the parabolic second order subjet of $u$ at $z_0\in M\times (0,T)$ is defined by
$$\mathcal{P}^{2,-} u (z_0):=-\mathcal{P}^{2,+} (-u) (z_0).$$
\end{definition}
We also define the closures of $\mathcal{P}^{2,+} u (z_0)$ and $\mathcal{P}^{2,-} u (z_0)$ by
\begin{align*}
\overline{\mathcal{P}}^{2,+}u(z_0)
&=\{(\tau,p,X)\in \R \times T_{x_0}M \times Sym^2(T^*_{x_0}M) |
\mbox{  there is a sequence  } (z_j,\tau_j,p_j,X_j) \\
&\mbox{  such that  } (\tau_j,p_j ,X_j)\in \mathcal{P}^{2,+}u(z_j) \\
&\mbox{  and  } (z_j,u(z_j),\tau_j,p_j,X_j) \to (z_0,u(z_0),\tau, p ,X) \mbox{  as  } j\to \infty \}; \\
\overline{\mathcal{P}}^{2,-}u(z_0)&=-\overline{\mathcal{P}}^{2,+}(-u)(z_0).
\end{align*}

Now we give the definition of a viscosity solution for the general equation
\begin{equation} \label{equ mfd}
u_t+F(x, t, u, Du, D^2 u)=0
\end{equation}
on $M$.
Assume $F\in C(M \times[0,T] \times \R \times T_{x_0}M \times Sym^2(T^*_{x_0}M))$ is proper, i.e.
$$F(x,t,r,p,X) \leq F(x,t,s,p,Y) \mbox{  whenever  } r\leq s, Y \leq X.$$


\begin{definition}
(i) A function $u \in \mbox{USC}(M\times(0,T))$ is a viscosity subsolution of \eqref{equ mfd}
if for all $z \in M\times(0,T)$ and $(\tau, p, X) \in \mathcal{P}^{2,+}u(z)$,
\begin{align*}
\tau +F(z, u(z), p, X) \leq 0.
\end{align*}

(ii) A function $u \in \mbox{LSC}(M\times(0,T))$ is a viscosity supersolution of \eqref{equ mfd}
if for all $z \in M\times(0,T)$ and $(\tau, p, X) \in \mathcal{P}^{2,-}u(z)$,
\begin{align*}
\tau +F(z, u(z) , p, X) \geq 0.
\end{align*}

(iii) A viscosity solution of \eqref{equ mfd} is defined to be a continuous function that is both a
viscosity subsolution and a viscosity supersolution of \eqref{equ mfd}.
\end{definition}

\subsection{Parabolic Maximum Principle for Semicontinuous Functions on Manifolds}
The main technical tool we use is the parabolic version maximum principle for semicontinous functions on manifolds, which is a restatement of \cite[Theorem 8.3]{CIL}, for Riemannian manifolds. One can also find it in \cite[Section 2.2]{I} or \cite[Theorem 3.8]{AFS1}.
\begin{thm}\label{max prin}
Let $M_1^{N_1}, \cdots, M_k^{N_k}$ be Riemannian manifolds, and $\Omega_i \subset M_i$ open subsets.
Let $u_i \in USC((0,T)\times \Omega_i)$, and $\vp$ defined on $(0,T)\times \Omega_1 \times \cdots \times \Omega_k$ such that $\vp$ is continuously differentiable in $t$ and twice continuously differentiable in
$(x_1, \cdots x_k) \in \Omega_1 \times \cdots \times \Omega_k$.
Suppose that $\hat{t} \in (0,T), \hat{x}_i \in \Omega_i$ for $i=1, \cdots, k$ and the function
$$\omega(t, x_1, \cdots, x_k) :=u_1(t,x_1)+\cdots + u_k(t,x_k)-\vp(t,x_1, \cdots , x_k) $$
attains a maximum at $(\hat{t},\hat{x}_1, \cdots, \hat{x}_k)$ on $(0,T)\times \Omega_1 \times \cdots \times \Omega_k$.
Assume further that there is an $r >0$  such that for every $M >0$ there is a $C>0$ such that for $i=1, \cdots, k$
\begin{align*}
& b_i \leq C  \mbox{  whenever  } (b_i,q_i,X_i) \in \ol{\mathcal{P}}^{2,+}u_i(t,x_i) , \\
& d(x_i, \hat{x}_i)+|t-\hat{t}| \leq r \mbox{  and  } |u_i(t,x_i)|+|q_i| +\|X_i\| \leq M.
\end{align*}
Then for each $\lambda>0$, there are $X_i \in Sym^2(T^*_{\hat{x}_i} M_i)$ such that
\begin{align*}
& (b_i,D_{x_i}\vp(\hat{t},\hat{x}_1, \cdots, \hat{x}_k),X_i)  \in \ol{\mathcal{P}}^{2,+}u_i(\hat{t},\hat{x}_i),\\
&  -\left(\frac 1 \lambda +\left\|M\right\| \right)I \leq
    \begin{pmatrix}
   X_1 & \cdots & 0 \\
   \vdots & \ddots & \vdots \\
   0 & \cdots & X_k
   \end{pmatrix}
   \leq M+\lambda M^2,  \\
& b_1 + \cdots + b_k =\vp_t(\hat{t},\hat{x}_1, \cdots, \hat{x}_k),
\end{align*}
where $M=D^2\vp(\hat{t},\hat{x}_1, \cdots, \hat{x}_k).$
\end{thm}

\section{Modulus of continuity estimates on manifolds}
For any given constant $\kappa $,  let
$$c_\kappa(t)=\left\{ \begin{matrix} \cos \sqrt{\kappa} t, & \, \kappa>0,\cr
                                   1, &\, \kappa=0, \cr
                                   \cosh \sqrt{|\kappa|}t, &\, \kappa<0.\end{matrix}\right.$$
Note that $c_{\kappa}(s)$ satisfies $c_{\kappa}'' + \kappa c_{\kappa} =0, c_{\kappa}(0)=1, c_{\kappa}'(0)=0$.
The following theorem is a generalization to viscosity solutions of Theorem 1 in \cite{AC3}.
\begin{theorem}\label{th1}
Let $u: M\times [0,T)\rightarrow \mathbb{R}$  be a viscosity solution of (\ref{1.1}) on a closed manifold
$M$ and denote by $D$ the diameter of $M$.
Assume further that $\text{Ric}_g\ge (n-1)\kappa g$. Then the modulus of continuity
$w: [0, \frac{D}{2}]\times [0,T)\rightarrow \mathbb{R}$ of $u$ satisfies
\begin{equation}
w_t\le \alpha(w', t)w''+(n-1)\frac{c_{\kappa}'(s)}{c_{\kappa}(s)} \beta(w', t)w' 
\end{equation}
in the viscosity sense, provided $\omega$ is increasing in $s$.
\end{theorem}

\begin{proof}
From the definition of viscosity solution, it suffices to show the following
\begin{itemize}
\item[]\emph{For any given $(s_0, t_0)$, a small neighborhood $U$ of $s_0$, $\epsilon_0>0$, and any smooth function
        $\phi$ lying above $w$ for $U\times (t_0-\epsilon_0, t_0+\epsilon_0)$ with equality at $(s_0, t_0)$, then
        \begin{equation}
        \phi_t\le \alpha(\phi', t_0)\phi''+(n-1)\frac{c_{\kappa}'}{c_{\kappa}}\beta(\phi',t_0)\phi' \label{2.2}
        \end{equation}
        holds at $(s_0, t_0)$.} 
\end{itemize}

Let $\phi$ be a smooth function lying above $w$ for $U\times (t_0-\epsilon_0, t_0+\e_0)$ with equality at $(s_0, t_0)$.
The assumption that $\omega$ is increasing in $s$ implies $\phi'(s_0,t_0) \geq 0$.
Since $M$ is compact, there exist $x_0$ and $y_0$ in $M$ with $d(x_0, y_0)=2s_0$ such that
$$u(y_0, t_0)-u(x_0, t_0)=2w(s_0,t_0)=2\phi(s_0,t_0).$$
Then it follows that
$$u(y,t)-u(x,t)-2\phi(\frac{d(x,y)}{2},t) $$
attains a local maximum at $(x_0,y_0,t_0)$. Note that the distance function $d$ may not be smooth at $(x_0,y_0)$, so one cannot apply the maximum principle for semicontinuous functions on manifolds directly.
To overcome this, we replace $d$ by a smooth function $\rho$, which is defined as follows.
Let $U_{x_0}$ and $U_{y_0}$ be small neighborhoods of $x_0$ and $y_0$ respectively.
Let $\gamma_0: [0,1]\rightarrow M$ be a minimizing geodesic joining $x_0$ and $y_0$ with $|\gamma_0'|=2s_0$.
We choose Fermi coordinates $ \left\{e_i(s)\right\}$ $(i=1,2,\cdots, n)$ along $\gamma_0$ with
$e_n(s)=\gamma_0'(s)$ for $s\in[0,1]$.
For $1\le i\le n-1$, we define $V_i(s)$ along $\gamma_0(s)$ by
$$V_i(s)=\frac{c_{\kappa}\left((2s-1)s_0\right)}{c_{\kappa}(s_0)}e_i(s),$$
and set $V_n(s)=e_n(s)$.
We then define a smooth function $\rho(x, y)$ in  $U_{x_0}\times U_{y_0}$ to be the length of the curve
$\exp_{\gamma_0(s)}\left(\sum_{i=1}^{n}\left((1-s)a_i(x)+sb_i(y)\right) V_i(s)\right)$ $(s\in [0,1])$,
where $a_i(x)$ and $b_i(x)$ are so defined that
$$x=\exp_{x_0}\left(\sum_{i=1}^{n}a_i(x)e_i(0)\right), \text{\quad\quad}
y=\exp_{y_0}\left(\sum_{i=1}^{n}b_i(y)e_i(1)\right).$$
From the definition of $\rho(x, y)$, we see $d(x,y)\le \rho(x, y)$ and with equality at $(x_0, y_0)$.
We write $\psi(x,y,t)=2\phi(\frac{\rho(x,y)}{2},t)$. Then the function
$$Z(y, x, t):=u(y,t)-u(x, t)-\psi(x,y,t)$$
has a local maximum at $(x_0,y_0,t_0)$.
Now we can apply the parabolic version maximum principle for semicontinuous functions on manifolds to conclude that
for each $\lambda >0$, there exist symmetric tensors $X, Y$ such that
\begin{equation*}
   (b_1, D_y \psi (x_0, y_0, t_0), X) \in \overline{\mathcal{P}}^{2,+} u(y_0,t_0),
  \end{equation*}
 \begin{equation*}
 (-b_2, - D_x \psi (x_0, y_0, t_0), Y) \in \overline{\mathcal{P}}^{2,-} u(x_0,t_0),
 \end{equation*}
 \begin{equation*}
 b_1+b_2= \psi_t (x_0, y_0, t_0)=2 \phi_t(s_0,t_0),
 \end{equation*}
  \begin{equation}
    \begin{pmatrix}
   X & 0 \\
   0 & -Y
   \end{pmatrix}
   \leq M+\lambda M^2,
  \end{equation}
 where $M=D^2 \psi(x_0,y_0,t_0)$.

The first derivative of $\psi$ yields
\begin{equation}
D_y \psi (x_0, y_0, t_0)=\phi'(s_0, t_0)\frac{\gamma'_0(1)}{2s_0}, \label{2.3}
\end{equation}
and
\begin{equation}
D_x \psi (x_0, y_0, t_0)=-\phi'(s_0, t_0)\frac{\gamma'_0(0)}{2s_0}. \label{2.4}
\end{equation}

Since $u$ is both a subsolution and a supersolution of (\ref{1.1}), we have
$$b_1\le \operatorname{tr}(A(|\phi'|)X)+b(|\phi'|, t_0),$$
and
$$-b_2\ge \operatorname{tr}(A(|\phi'|)Y)+b(|\phi'|, t_0),$$
where $$
A=\left(
    \begin{array}{cccc}
      \beta(|\phi'|, t_0) & \cdots & 0 & 0 \\
     \vdots & \vdots & \vdots & \vdots \\
     0 & \cdots &\beta(|\phi'|,t_0) & 0 \\
      0 & \cdots & 0 & \alpha(|\phi'|, t_0) \\
    \end{array}
  \right).
$$
Set
$$C=\left(
    \begin{array}{cccc}
      \beta(|\phi'|, t_0) & \cdots & 0 & 0 \\
     \vdots & \vdots & \vdots & \vdots \\
     0 & \cdots &\beta(|\phi'|,t_0) & 0 \\
      0 & \cdots & 0 & -\alpha(|\phi'|, t_0) \\
    \end{array}
  \right),$$
and simple calculation shows $\left(
                           \begin{array}{cc}
                             A & C \\
                            C & A \\
                           \end{array}
                         \right)\ge 0$.
Then we obtain that
\begin{eqnarray*}
2\phi_t(s_0, t_0)&=&b_1+b_2\le \operatorname{tr}\left[\left(
                                                        \begin{array}{cc}
                                                          A & C \\
                                                         C & A \\
                                                        \end{array}
                                                      \right)
\left(
\begin{array}{cc}
 X & 0 \\
  0 & -Y \\
   \end{array}
   \right)
\right]\nonumber\\
&\le&\operatorname{tr}\left[\left(
                                                        \begin{array}{cc}
                                                          A & C \\
                                                         C & A \\
                                                        \end{array}
                                                      \right)M
\right]
+\lambda\operatorname{tr}\left[\left(
                                                        \begin{array}{cc}
                                                          A & C \\
                                                         C & A \\
                                                        \end{array}
                                                      \right)M^2
\right].\nonumber\\
\end{eqnarray*}
We easily get
\begin{eqnarray}
\operatorname{tr}\left[\left(
                                                        \begin{array}{cc}
                                                          A & C \\
                                                         C & A \\
                                                        \end{array}
                                                      \right)M
\right]
&=&\alpha(|\phi'|, t_0)D^2\psi\left((e_n(1),-e_n(0)), (e_n(1),-e_n(0))\right)\\ \nonumber
&& +\beta(|\phi'|, t_0)\sum_{i=1}^{n-1}D^2\psi\left((e_i(1),e_i(0)), (e_i(1),e_0(0))\right) \label{2.6}
\end{eqnarray}

It remains to estimate the terms involving second derivatives of $\psi$.
The estimate is analogous to \cite[Theorem 3]{AC3}.
For $1\le i\le n-1$, we choose the variation vector fields $V_i(s)=\frac{c_{\kappa}\left((2s-1)s_0\right)}{c_{\kappa}(s_0)}e_i(s)$ along $\gamma_0(s)$,
then the first variation formulas gives
$$
\left.\frac{d}{dv}\right|_{v=0}|\gamma_v|=\frac{1}{2s_0} g(\gamma', V_i)|^1_0=0,
$$
and the second variation formula gives
\begin{equation}
\left.\frac{d^2}{dv^2}\right|_{v=0}|\gamma_v|=\left.\frac{1}{2s_0} g(\gamma', \nabla_{V_i} V_i)\right.|^1_0+\frac{1}{2s_0}
\int_0^1|(\nabla_{\gamma'} V_i)^{\bot}|^2-\langle R(\gamma', V_i)\gamma', V_i\rangle\ ds. \label{2.7}
\end{equation}
By the way of variation, we can also require $\nabla_{V_i} V_i=0$ for $s\in [0,1]$.
Therefore direct calculation gives
$$
\frac{1}{2s_0}\int_0^1|(\nabla_{\gamma'} V_i)^{\bot}|^2 ds=2s_0\int_0^1(\frac{c_{\kappa}'\left((2s-1)s_0\right)}{c_{\kappa}(s_0)})^2ds=
\int_{-s_0}^{s_0}(\frac{c_{\kappa}'(x)}{c_{\kappa}(s_0)})^2dx.
$$
Using the integration by parts, the definition of $c_\kappa$ and equation $c_{\kappa}''+\kappa c_\kappa=0$, we have
$$
\int_{-s_0}^{s_0}(\frac{c_{\kappa}'(x)}{c_{\kappa}(s_0)})^2dx=2\frac{c_{\kappa}'(s_0)}{c_{\kappa}(s_0)}
+\int_{-s_0}^{s_0}\kappa(\frac{c_{\kappa}(x)}{c_{\kappa}(s_0)})^2dx.
$$
Combining with (\ref{2.7}), we see
$$
\left.\frac{d^2}{dv^2}\right|_{v=0}|\gamma_v|=2\frac{c_{\kappa}'(s_0)}{c_{\kappa}(s_0)}
+\int_{-s_0}^{s_0}(\frac{c_{\kappa}(x)}{c_{\kappa}(s_0)})^2 (\kappa-\langle R(e_n, e_i)e_n, e_i\rangle)dx.
$$

Then we conclude 
that
\begin{eqnarray}
&&\frac 1 2 D^2\psi\left((e_i(1),e_i(0)), (e_i(1),e_i(0))\right)\nonumber \\
&=&\left.\frac{d^2}{dv^2}\right|_{v=0} \phi \left( \frac 1 2 \rho\left(\exp_{x_0} ve_i(0), \exp_{y_0} ve_i(1)\right),t_0 \right) \nonumber \\
&=& \left.\frac{d^2}{dv^2}\right|_{v=0} \phi \left(\frac 1 2 L\left[\left(\exp_{\gamma_0(s)} \left((1-s)v+sv\right) V_i(s)\right)\right],t_0\right) \nonumber \\
&=& \left.\frac{d^2}{dv^2}\right|_{v=0} \phi \left(\frac 1 2 L\left[\left(\exp_{\gamma_0(s)} v V_i(s)\right)\right],t_0\right)\nonumber \\
&=&\left.\frac{d^2}{dv^2}\right|_{v=0}\phi(\frac{|\gamma_v|}{2}, t_0) \nonumber \\
&=&\phi'(s_0, t_0)\left(\frac{c_{\kappa}'(s_0)}{c_{\kappa}(s_0)}
+\frac{1}{2}\int_{-s_0}^{s_0}(\frac{c_{\kappa}(x)}{c_{\kappa}(s_0)})^2 (\kappa-\langle R(e_n, e_i)e_n, e_i\rangle)dx\right),\label{3.7}
\end{eqnarray}
from which, the following holds
\begin{eqnarray}
&&\sum_{i=1}^{n-1}\frac 1 2 D^2\psi\left((e_i(1),e_i(0)), (e_i(1),e_0(0))\right)\nonumber\\
&=&(n-1)\phi'(s_0, t_0)\left(\frac{c_{\kappa}'(s_0)}{c_{\kappa}(s_0)}
+\frac{1}{2}\int_{-s_0}^{s_0}(\frac{c_{\kappa}(x)}{c_{\kappa}(s_0)})^2 ((n-1)\kappa-\operatorname{Ric}(e_n, e_n))\ dx\right)\nonumber\\
&\le&(n-1)\frac{c_{\kappa}'(s_0)}{c_{\kappa}(s_0)}\phi'(s_0, t_0), \label{2.8}
\end{eqnarray}
where we used the curvature assumption.

It follows easily from the variation along $e_n$ that
\begin{equation}
\frac 1 2 D^2\psi\left((e_n(1),-e_n(0)), (e_n(1),-e_n(0))\right)=\phi''(s_0, t_0). \label{2.9}
\end{equation}
Thus we conclude from (\ref{2.6}), (\ref{2.8}) and (\ref{2.9}) that
\begin{eqnarray}
\operatorname{tr}\left[\left(
                                                        \begin{array}{cc}
                                                          A & C \\
                                                         C & A \\
                                                        \end{array}
                                                      \right)M
\right]\le 2\beta(\phi', t_0)(n-1)\frac{c_{\kappa}'(s_0)}{c_{\kappa}(s_0)}\phi'(s_0, t_0)+2\alpha(\phi', t_0)\phi''(s_0, t_0). \nonumber\\
\label{2.10}
\end{eqnarray}
Since $\lambda>0$ is arbitrary, (\ref{2.2}) comes true from (\ref{2.6}) and (\ref{2.10}). Hence we complete the proof.
\end{proof}

As an immediate corollary, we have the following Ricci flow version,
which generalizes Theorem 4 in \cite{A}.
\begin{thm}[Ricci flow version]\label{thr}
Let $M^n$ be a closed Riemannian manifold, and $g(t)$ a family of time-dependent metrics on $M$
satisfying $\frac{\partial g}{\partial t}\ge-2\text{Ric} $, and let $u: M\times [0,T)\rightarrow \mathbb{R}$
be a viscosity solution of (\ref{1.1}). Then the modulus of continuity $w:[0, \frac{D}{2}]\times [0,T) \to \R$ of $u$
satisfies
$$w_t\le \alpha(w', t)w''$$
in the viscosity sense, provided $\omega$ is increasing in $s$.
\end{thm}
\begin{proof}
As before, we consider a smooth function $\phi$ which lies above the modulus of continuity $w$ and
is equal at $(s_0, t_0)$. Then via maximum principle as before, it holds that
\begin{eqnarray*}
2\phi_t(s_0, t_0)-\int_{-s_0}^{s_0} \operatorname{Ric}(e_n(s), e_n(s))\ ds\le b_1+b_2
\le \operatorname{tr}\left[\left(
                                                        \begin{array}{cc}
                                                          A & C \\
                                                         C & A \\
                                                        \end{array}
                                                      \right)M
\right].
\end{eqnarray*}
On the other hand, choosing the variation fields $V_i(s)=e_i(s)$ yields
$$\operatorname{tr}\left[\left(
                                                        \begin{array}{cc}
                                                          A & C \\
                                                         C & A \\
                                                        \end{array}
                                                      \right)M
\right]\le 2\alpha(\phi', t_0)\phi''(s_0, t_0)-\int_{-s_0}^{s_0} \operatorname{Ric}(e_n(s), e_n(s))\ ds,$$
completing the proof. Here we used inequality in (\ref{2.8}) with $c_\kappa=1$ and inequality (\ref{2.9}).
\end{proof}

\section{Height-dependent gradient bounds}
In this section, we obtain height-dependent gradient bounds for viscosity solutions, generalizing Theorem 6 in \cite{A}.
\begin{theorem}\label{thmh}
Let $(M^n, g)$ be a closed Riemannian manifold with diameter $D$ and Ricci curvature satisfying
$\text{Ric}_g\ge 0$, and suppose  $u: M\times [0,T)\rightarrow \mathbb{R}$ is a viscosity solution of an
equation of the form
\begin{equation}
\frac{\partial u}{\partial t}=\left[\alpha(|D u|,u, t)\frac{D_iu D_ju}{|Du|^2}
+ \beta(t)\left(\delta_{ij}-\frac{D_iu D_ju}{|Du|^2}\right)\right]D_iD_ju .\label{c1}
\end{equation}
Let $\varphi: [0,D]\times [0,T)\rightarrow \mathbb{R}$ be a solution of
$$\varphi_t=\alpha(\varphi', \vp, t)\varphi'' ,$$
with Neumann boundary condition, which is increasing in the first variable, such that the
range of $u(\cdot, 0)$ is contained in $[\varphi(0,0), \varphi(D,0)]$. Let $\Psi(s,t)$ be given by
inverting $\varphi$ for each $t$, and assume that for all $x$ and $y$ in $M$,
$$\Psi(u(y,0),0)-\Psi(u(x,0),0)-d(x,y)\le 0.$$
Then
$$
\Psi(u(y,t),t)-\Psi(u(x,t),t)-d(x,y)\le 0.
$$
for all $x,y\in M$ and $t\in[0,T)$.
\end{theorem}

We begin with a lemma about the behavior of parabolic semijets when composed with an increasing function.

\begin{lemma}\label{Lemma}
Let $u$ be a continuous function. Let $\vp:\R \times [0, T) \to \R $ be a $C^{2,1}$ function with $\vp^{\prime} \geq 0$.
Let $\Psi:\R \times [0, T) \to \R $ be such that
$$\Psi(\vp(u(y,t),t),t)=u(y,t)$$
$$\vp (\Psi(u(y,t),t),t)=u(y,t)$$
(i) Suppose $(\tau, p, X) \in \mathcal{P}^{2,+}\Psi(u(y_0,t_0),t_0)$, then
$$(\vp_t+\vp^{\prime}\tau, \vp^{\prime}p, \vp^{\prime \prime}p \otimes p+ \vp^{\prime} X ) \in \mathcal{P}^{2,+}  u(y_0,t_0),$$
where all derivatives of $\vp$ are evaluated at $(\Psi(u(y_0,t_0)), t_0)$.

(ii) Suppose $(\tau, p, X) \in \mathcal{P}^{2,-}\Psi(u(y_0,t_0),t_0)$, then
$$(\vp_t+\vp^{\prime}\tau, \vp^{\prime}p, \vp^{\prime \prime}p \otimes p+ \vp^{\prime} X ) \in \mathcal{P}^{2,-}  u(y_0,t_0),$$
where all derivatives of $\vp$ are evaluated at $(\Psi(u(y_0,t_0)), t_0)$.

(iii) The same holds if one replaces the parabolic semijets by the their closures.
\end{lemma}

\begin{proof}
The lemma is an easy consequence of the following characterization of the semijets.
\begin{align*}
\mathcal{P}^{2,+}u(y_0,t_0)&=\{\left(\vp_t(y_0, t_0), D\vp (y_0, t_0), D^2 \vp(y_0, t_0)\right) | \\
& \vp \in C^{2,1} \mbox{ and } u-\vp \mbox{ has a local maximum at } (y_0,t_0) \}
\end{align*}
\begin{align*}
\mathcal{P}^{2,-}u(y_0,t_0)&=\{\left(\vp_t(y_0, t_0), D\vp (y_0, t_0), D^2 \vp(y_0, t_0)\right) | \\
& \vp \in C^{2,1} \mbox{ and } u-\vp \mbox{ has a local mimimum at } (y_0,t_0) \}
\end{align*}
For (i), Suppose $(\tau, p, X) \in \mathcal{P}^{2,+}\Psi(u(y_0,t_0),t_0)$. Let $h$ be a $C^{2,1}$ function such that $\Psi(u(y,t),t)-h(y,t)$ has a local max at $(y_0,t_0)$ and $(h_t, Dh, D^2h)(y_0,t_0)=(\tau , p, X)$.
Since $\vp$ is increasing, we have $u(y,t)-\vp(h(y,t),t)=\vp (\Psi(u(y,t),t), t)-\vp(h(y,t),t)$ has a local max at $(y_0,t_0)$
Then it follows that $$(\vp_t+\vp^{\prime}\tau, \vp^{\prime}p, \vp^{\prime \prime}p \otimes p+ \vp^{\prime} X) \in \mathcal{P}^{2,+}u(y_0,t_0).$$

 For (ii), Suppose $(\tau, p, X) \in \mathcal{P}^{2,-}\Psi(u(y_0,t_0),t_0)$. Let $h$ be a $C^{2,1}$ function such that $\Psi(u(y,t),t)-h(y,t)$ has a local min at $(y_0,t_0)$ and $(h_t, Dh, D^2h)(y_0,t_0)=(\tau , p, X)$.
 Since $\vp$ is increasing, we have $u(y,t)-\vp(h(y,t),t)=\vp (\Psi(u(y,t),t), t)-\vp(h(y,t),t)$ has a local min at $(y_0,t_0)$
 Then it follows that $$(\vp_t+\vp^{\prime}\tau, \vp^{\prime}p, \vp^{\prime \prime}p \otimes p+ \vp^{\prime} X) \in \mathcal{P}^{2,-}u(y_0,t_0).$$

 (iii) then follows from approximation.
\end{proof}
\begin{proof}[Proof of Theorem \ref{thmh}]
The theorem is valid if we show that for any $\e>0$,
\begin{equation}
Z^{\e}(x,y,t):=\Psi(u(y,t),t)-\Psi(u(x,t),t)-d(x,y)-\frac{\e}{T-t}\le 0.\label{Ze}
\end{equation}
To prove inequality (\ref{Ze}), it suffices to show $Z^{\e}$ can not attain the maximum in $M\times M\times(0, T)$.
Assume by contradiction that there exist  $t_0 \in (0,T)$, $x_0$ and $y_0$ in $M$ at which the function
$Z^{\e}$
attains its maximum. Take $\rho$ defined as before.   Then the function
$$\Psi(u(y,t),t)-\Psi(u(x,t),t)-\rho(x,y)-\frac{\e}{T-t}$$ has a local maximum at $(x_0,y_0,t_0)$.
If $\e >0$, then we necessarily have $x_0 \neq y_0$.
By the parabolic maximum principle for semicontinuous functions on manifolds,
for any $\lambda >0$, there exist $X, Y$ satisfying

  \begin{equation*}
   (b_1, D_y \rho(x_0, y_0), X) \in \overline{\mathcal{P}}^{2,+} \Psi(u(y_0,t_0), t_0),
  \end{equation*}
 \begin{equation*}
 (-b_2, - D_x \rho(x_0, y_0), Y) \in \overline{\mathcal{P}}^{2,-} \Psi(u(x_0,t_0), t_0),
 \end{equation*}
 \begin{equation*}
 b_1+b_2=\frac{\e}{(T-t_0)^2},
 \end{equation*}
\begin{equation}\label{Hessian inequality for quasilinear}
  -\left(\lambda^{-1}+\left\|M\right\| \right)I \leq
    \begin{pmatrix}
   X & 0 \\
   0 & -Y
   \end{pmatrix}
   \leq M+\lambda M^2,
  \end{equation}
 where $M=D^2 \rho(x_0,y_0)$.
By Lemma \ref{Lemma}, we have
\begin{equation*}
   (b_1 \vp^{\prime}(z_{y_0}, t_0)+\vp_t(z_{y_0}, t_0), \vp^{\prime}(z_{y_0}, t_0) D_y \rho(s_0, t_0), \vp^{\prime}(z_{y_0}, t_0)X+\vp''(z_{y_0}, t_0)e_n(1)\otimes e_n(1)) \in \overline{\mathcal{P}}^{2,+} u(y_0,t_0),
  \end{equation*}
  and
 \begin{equation*}
 (-b_2 \vp^{\prime}(z_{x_0}, t_0) +\vp_t(z_{x_0}, t_0), -\vp^{\prime} (z_{x_0}, t_0) D_x \rho(s_0, t_0), \vp^{\prime}(z_{x_0}, t_0) Y+\vp''(z_{x_0}, t_0)e_n(0)\otimes e_n(0)) \in \overline{\mathcal{P}}^{2,-} u(x_0,t_0), \label{3.3}
 \end{equation*}
where
$z_{x_0}=\Psi(u(x_0, t_0), t_0)$, $z_{y_0}=\Psi(u(y_0, t_0), t_0)$ and $e_n(s)$ is as defined in Section 2.

Since $u$ is both a subsolution and a supersolution of (\ref{c1}), we have
$$
b_1 \vp^{\prime}(z_{y_0}, t_0)+\vp_t(z_{y_0}, t_0)\le \operatorname{tr}\left(\vp^{\prime}(z_{y_0}, t_0) A_1 X+
\vp''(z_{y_0}, t_0)A_1e_n(1)\otimes e_n(1)\right),
$$
and
$$
-b_2 \vp^{\prime}(z_{x_0}, t_0) +\vp_t (z_{x_0}, t_0)\ge \operatorname{tr}( \vp^{\prime}(z_{x_0}, t_0)A_2Y+\vp''(z_{x_0}, t_0)A_2 e_n(0)\otimes e_n(0)),
$$
where $$
A_1=\left(
    \begin{array}{cccc}
      \beta( t_0) & \cdots & 0 & 0 \\
     \vdots & \vdots & \vdots & \vdots \\
     0 & \cdots &\beta(t_0) & 0 \\
      0 & \cdots & 0 & \alpha(|\vp^{\prime}(z_{y_0}, t_0)|,\vp(z_{y_0}, t_0), t_0) \\
    \end{array}
  \right),
$$
and
$$
A_2=\left(
    \begin{array}{cccc}
      \beta( t_0) & \cdots & 0 & 0 \\
     \vdots & \vdots & \vdots & \vdots \\
     0 & \cdots &\beta(t_0) & 0 \\
      0 & \cdots & 0 & \alpha(|\vp^{\prime}(z_{x_0}, t_0)|,\vp(z_{x_0}, t_0), t_0) \\
    \end{array}
  \right).
$$
Set
$$C=\left(
    \begin{array}{cccc}
      \beta(t_0) & \cdots & 0 & 0 \\
     \vdots & \vdots & \vdots & \vdots \\
     0 & \cdots &\beta(t_0) & 0 \\
      0 & \cdots & 0 & 0 \\
    \end{array}
  \right),$$
and simple calculation shows $\left(
                           \begin{array}{cc}
                             A_1 & C \\
                            C & A_2 \\
                           \end{array}
                         \right)\ge 0$.
Then we obtain  that
\begin{eqnarray*}
\frac{\e}{(T-t_0)^2} &=&b_1+b_2 \le \frac{\operatorname{tr}(\vp^{\prime}(z_{y_0}, t_0)A_1X)-
\vp_t(z_{y_0}, t_0)}{\vp^{\prime}(z_{y_0}, t_0)}+\frac{\operatorname{tr}(A_1 e_n(1)\otimes e_n(1))
\vp''(z_{y_0}, t_0)}{\vp^{\prime}(z_{y_0}, t_0)}\\
&&+\frac{\operatorname{tr}(-\vp^{\prime}(z_{x_0}, t_0)A_2 Y)+
\vp_t(z_{x_0}, t_0)}{\vp^{\prime}(z_{x_0}, t_0)}-\frac{\operatorname{tr}(A_2 e_n(0)\otimes e_n(0))
\vp''(z_{x_0}, t_0)}{\vp^{\prime}(z_{x_0}, t_0)}\\
&&+\lambda\operatorname{tr}\left[\left(
                                                        \begin{array}{cc}
                                                      A_1 & C \\
                                                         C &  A_2 \\
                                                        \end{array}
                                                      \right)M^2
\right]\\
&\le& \operatorname{tr}\left[\left(
                                                        \begin{array}{cc}
                                                          A_1& C \\
                                                         C &  A_2 \\
                                                        \end{array}
                                                      \right)
\left(
\begin{array}{cc}
 X & 0 \\
  0 & -Y \\
   \end{array}
   \right)
\right]
+\lambda\operatorname{tr}\left[\left(
                                                        \begin{array}{cc}
                                                      A_1 & C \\
                                                         C &  A_2 \\
                                                        \end{array}
                                                      \right)M^2
\right] \\
&&+\frac{\vp_t(z_{x_0}, t_0)-\alpha(\vp'(z_{x_0},t_0),\vp(z_{x_0}, t_0), t_0)\vp''(z_{x_0},t_0)}{\vp^{\prime}(z_{x_0}, t_0)}\\
&&-\frac{\vp_t(z_{y_0}, t_0)-\alpha(\vp'(z_{y_0},t_0),\vp(z_{y_0}, t_0),t_0)\vp''(z_{y_0},t_0)}{\vp^{\prime}(z_{y_0}, t_0)}\\
&=&\operatorname{tr}\left[\left(
                                                        \begin{array}{cc}
                                                       A_1& C \\
                                                         C &  A_2 \\
                                                        \end{array}
                                                      \right)M
\right]
+\lambda\operatorname{tr}\left[\left(
                                                        \begin{array}{cc}
                                                      A_1 & C \\
                                                         C &  A_2 \\
                                                        \end{array}
                                                      \right)M^2
\right].
\end{eqnarray*}
Using (\ref{2.8}) with $\kappa=0$, we obtain
\begin{eqnarray*}
\operatorname{tr}\left[\left(
                                                        \begin{array}{cc}
                                                        A_1& C \\
                                                         C &  A_2\\
                                                        \end{array}
                                                      \right)M
\right]\leq 0.
\end{eqnarray*}

We have arrived at $ \e \leq 0$ by letting $\lambda \to 0$, which is a contradiction. Therefore (\ref{Ze}) is true, hence completing the proof.
\end{proof}

\section{Estimate on Bakry-Emery Manifolds}
We prove a generalization to viscosity solutions of Theorem 1.2 in \cite{AN}.
\begin{thm}
Let $M$ be a closed Riemannian manifold satisfying
$$Ric_{ij}+f_{ij} \geq a g_{ij}$$
for some $a\in \mathbb{R}$.
Denote $D = diam(M)$. Let $u$ be a viscosity solution of
$$u_t = \Delta_f u$$
with the operator $\Delta_f:=\Delta -\langle \nabla (\cdot), \nabla f \rangle$.
Then the modulus of continuity $\omega :[0,\frac{D}{2}] \times \R_{+}\to \R$ of $u$ is a viscosity subsolution of
$$\omega_t = \omega''-a s \omega',$$
provided $\omega$ is increasing in $s$.
\end{thm}

\begin{proof}
The idea is the same as before. Let $\phi$ be a smooth function lying above $w$ for $U\times (t_0-\epsilon_0, t_0+\e_0)$ with equality at $(s_0, t_0)$.
Then it follows that the function
$$Z(y, x, t):=u(y,t)-u(x, t)-\psi(x,y,t)$$
has a local maximum at $(x_0,y_0,t_0)$, where $\psi(x,y,t)=\phi\left(\frac{\rho(x,y)}{2},t\right)$ as before.
Now we can apply the parabolic version maximum principle for semicontinuous functions on manifolds to conclude that
for each $\lambda >0$, there exist symmetric tensors $X, Y$ such that
\begin{equation*}
   (b_1, D_y \psi (x_0, y_0, t_0), X) \in \overline{\mathcal{P}}^{2,+} u(y_0,t_0),
  \end{equation*}
 \begin{equation*}
 (-b_2, - D_x \psi (x_0, y_0, t_0), Y) \in \overline{\mathcal{P}}^{2,-} u(x_0,t_0),
 \end{equation*}
 \begin{equation*}
 b_1+b_2= \psi_t (x_0, y_0, t_0)=2 \phi_t(s_0,t_0),
 \end{equation*}
  \begin{equation}
    \begin{pmatrix}
   X & 0 \\
   0 & -Y
   \end{pmatrix}
   \leq M+\lambda M^2,
  \end{equation}
 where $M=D^2 \psi(x_0,y_0,t_0)$.

The first derivative of $\psi$ yields
\begin{equation}
D_y \psi (x_0, y_0, t_0)=\phi'(s_0, t_0)\frac{\gamma'_0(1)}{2s_0}=\phi'(s_0, t_0) e_n(1), \label{2.3}
\end{equation}
and
\begin{equation}
D_x \psi (x_0, y_0, t_0)=-\phi'(s_0, t_0)\frac{\gamma'_0(0)}{2s_0}=\phi'(s_0, t_0) e_n(0). \label{2.4}
\end{equation}
Since $u$ is a viscosity solution, we have
$$b_1 \leq \tr(X)-\phi'\langle\nabla f(y_0), e_n(1)\rangle,$$
$$-b_2 \geq \tr(Y)-\phi'\langle\nabla f(x_0), e_n(0)\rangle$$
Therefore
\begin{align*}
2\phi_t(s_0,t_0)=b_1 +b_2 & \leq \tr(X)-\tr(Y) -\phi'\langle\nabla f(y_0), e_n(1)\rangle +\phi'\langle\nabla f(x_0), e_n(0)\rangle \\
& \leq \tr(M)  -\phi'\langle\nabla f(y_0), e_n(1)\rangle +\phi'\langle\nabla f(x_0), e_n(0)\rangle.
\end{align*}
We estimate that
\begin{align*}
\tr(M) &= \sum_{i=1}^{n-1} D^2\phi\left((e_i(1),e_i(0)), (e_i(1),e_i(0)) \right) +D^2\phi\left((e_n(1),-e_n(0)), (e_n(1),-e_n(0))\right)\\
& \leq \phi''-\phi'\int_{-s_0}^{s_0} \Ric(e_n(s),e_n(s)) ds \\
&\leq \phi'' +\phi' \int_{-s_0}^{s_0}  \nabla \nabla f(e_n,e_n) -a  g\langle e_n,e_n \rangle ds \\
& = \phi''+2 a s_0 \phi' + \phi'\langle\nabla f(y_0), e_n(1)\rangle +\phi'\langle\nabla f(x_0), e_n(0)\rangle,
\end{align*}
where we have used \eqref{3.7} with $\kappa=0$.
Hence at $(s_0,t_0)$, $$\phi_t \leq \phi''- a s_0 \phi' $$
holds, proving the theorem.
\end{proof}

Next we give the evolutionary analogue of the above theorem, which is a generalization of Theorem 5 in \cite{A}.
\begin{thm}
Let $M^n$ be a closed manifold with time-dependent metrics and smooth function $f(\cdot,t)$. Suppose that
$$g_t \geq -2\left(Ric_{ij}+f_{ij} \right)+ 2 a g_{ij}.$$
Let $u$ be a viscosity solution of the drift-Laplacian heat flow
$$u_t= \Delta_f u.$$
Then the modulus of continuity $\omega :[0,\frac{D}{2}] \times \R_{+}\to \R$ of $u$ is a viscosity subsolution of
$$\omega_t = \omega''-a s \omega'$$
provided $\omega$ is increasing in $s$.
\end{thm}
\begin{proof}
The proof is immediate by combining
$$\phi_t -\int_{-s_0}^{s_0} \left(\Ric(e_n,e_n)+\nabla^2 f(e_n,e_n) -ag\langle e_n,e_n \rangle \right)\ ds  \leq b_1+b_2,$$
$$b_1 +b_2 \leq \tr(M)  -\phi'\langle\nabla f(y_0), e_n(1)\rangle +\phi'\langle\nabla f(x_0), e_n(0)\rangle, $$
and
$$\tr (M)\leq \phi''-\phi'\int_{-s_0}^{s_0} \Ric(e_n(s),e_n(s))\ ds.$$
\end{proof}
\begin{remark}
The same proof works for the more general equation: If $u$ is a viscosity solutions of
$$u_t = \left[\alpha(|Du|,t) \frac{D_iu D_ju}{|Du|^2}+\beta(|Du|,t)\left(\delta_{ij}-\frac{D_iu D_ju}{|Du|^2}\right)\right]D_iD_ju +b(|Du|, t)+\langle \nabla f, \nabla u \rangle.$$
Then the modulus of continuity $\omega$ of $u$ satisfies
 $$\omega_t \leq \alpha(\omega ') \omega''-a s \omega'$$
 in the viscosity sense, provided $\omega$ is increasing in $s$.
\end{remark}

\section{Boundary Value Problems}

\subsection{Definition of Viscosity Solution for Boundary Value Problems}
\mbox{  } \\
We recall the definition of viscosity solutions to boundary value problems from \cite[Section 7]{CIL}.
Let $\Omega$ be an open subset of $\R^n$ and $T>0$. For brevity we write $z=(x,t)$. 
Consider the boundary problem of the form
\begin{eqnarray}\label{BVP}
\mbox{     }u_t+F(z,u,Du,D^2u)=0 \mbox{  in  } \Omega \times (0,T), \mbox{       }
B(z,u,Du,D^2u) =0 \mbox{  on   } \p \Omega \times (0,T).
\end{eqnarray}
Assume $F\in C(\ol{\Omega} \times \R \times \R \times \R^n \times S^{n\times n})$ and $B \in C(\p \Omega \times \R \times \R \times \R^n \times S^{n\times n})$
are both proper.

\begin{definition}
A function $u\in USC(\ol{\Omega}\times (0,T))$ is a viscosity subsolution of  \eqref{BVP} if
$$\tau+F(z,u(z),p,X) \leq 0 \mbox{   for   }  z \in \Omega \times(0,T),
(\tau, p ,X ) \in  \overline{\mathcal{P}}^{2,+}_{\ol{\Omega}\times (0,T)} u(z),$$
and
$$\min \left\{\tau+F(z,u(z),p,X), B(z, u(z),p,X) \right\} \leq 0
\mbox{   for   }  z \in \p \Omega \times(0,T), (\tau, p ,X ) \in  \overline{\mathcal{P}}^{2,+}_{\ol{\Omega}\times (0,T)} u(z).$$

Similarly, $u\in LSC(\ol{\Omega}\times (0,T))$ is a viscosity supersolution of  \eqref{BVP} if
$$\tau+F(z,u(z),p,X) \geq 0 \mbox{   for   }  z \in \Omega \times(0,T),
(\tau, p ,X ) \in  \overline{\mathcal{P}}^{2,-}_{\ol{\Omega}\times (0,T)} u(z),$$
and
$$\max \left\{\tau+F(z,u(z),p,X), B(z,u(z),p,X) \right\} \geq 0
\mbox{   for   }  z \in \p \Omega \times(0,T), (\tau, p ,X ) \in  \overline{\mathcal{P}}^{2,-}_{\ol{\Omega}\times (0,T)} u(z).$$

Finally, $u$ is a viscosity solution of \eqref{BVP} if it is both a viscosity subsolution and a viscosity supersolution of \eqref{BVP}.
\end{definition}

\subsection{Neumann problem}

We consider the following quasilinear evolution equations:
\begin{align*}
\frac{\p u}{\p t}=a^{ij}(Du,t) D_iD_j u + b(Du, t)\mbox{  in   } \Omega \times (0,T),  \\
\langle Du(x,t), n(x) \rangle=0 \mbox{  on  } \p \Omega\times (0,T),
\end{align*}
where $A(p,t)=\left(a^{ij}(p, t)\right)$ is positive semi-definite and $n(x)$ is the exterior unit normal vector at $x$.
As in \cite{AC2}, we assume that there exists a continuous function
$\alpha: \R_{+}\times [0, T] \to \R$ with
\begin{equation}
0 < \alpha(R, t) \leq R^2 \inf_{|p|=R, (v \cdot p) \neq 0} \frac{v^{T}A(p, t)v}{(v \cdot p)^2},
\end{equation}

The following theorem is a generalization of Theorem 4.1 in \cite{AC2} to viscosity solutions.
\begin{thm}
Let $\Omega \subset \RR^n$ be a smooth bounded and convex domain.
Let $u$ be a viscosity solution of the Neumann problem.
Then the modulus of continuity $\omega$ is a viscosity subsolution of the one dimensional equation
$\omega_t = \alpha(|\omega^{\prime}|, t) \omega^{\prime \prime}$, provided that $\omega$ is increasing.
\end{thm}

\begin{proof}
We must show that if $\phi$ is a smooth function such that $\omega -\phi$ has a local maximum at $(s_0,t_0)$
 for $s_0>0$ and $t_0>0$, then  at $(s_0,t_0)$
 $$\phi_t \leq \alpha(|\phi^{\prime}|, t) \phi^{\prime \prime}.$$
 As before, we consider the function
 $$Z(y,x,t)=u(y,t)-u(x,t)-2\phi(\frac{|y-x|}{2},t)$$
 and arrive at that there exists $(x_0,y_0,t_0)$ with $|x_0-y_0|=2s_0$
 such that $Z$ attains a local maximum at $(x_0,y_0,t_0)$.
 Now replacing $\phi$ by
   $$\vp(s,t)=\phi(s,t)+(s-s_0)^4+(t-t_0)^4$$
   if necessary, we may assume $Z$ has a strict local maximum at $(x_0,y_0,t_0)$.

If $(x_0,y_0)\in \Omega \times \Omega$, then the same argument as in \cite{me1} would prove the theorem.
For the case $(x_0,y_0)\in \p (\Omega \times \Omega)$, the strategy  is to produce approximations $u^\e, u_\e$
 such that $u^\e \to u$, $u_\e \to u$ uniformly as $\e \to 0$
 and $u^\e, u_\e$ are a supersolution and a subsolution of some modified equation, for which we have the same inequalities no matter the maximum point lies in $\Omega\times \Omega$ or on $\p (\Omega \times \Omega)$.

Fix a point $z_0\in \Omega$. Let $\delta = d(z_0, \p \Omega)>0$.
 Define $v(x)=\frac 1 2 (x-z_0)^2$. Then $Dv(x)=x-z_0$ and $D^2v(x)=I$.
 Moreover for any $x\in \p \Omega$, $$\langle Dv(x), n(x) \rangle = \langle x-z_0, n(x) \rangle \geq d(z_0, \p \Omega) =\delta. $$
 Define $u_\e(x) = u(x)-\e v(x)$ and $u^\e(x) = u(x)+\e v(x)$. Then we have the following:\\
$u_\e$ is a viscosity subsolution of
 \begin{align}
 \frac{\p u}{\p t}=a^{ij}(Du+\e Dv,t) D_iD_j u+b(Du+\e Dv,t)+\e \operatorname{tr}(a^{ij}(Du+\e Dv,t)) \mbox{  in   } \Omega,  \\
 \langle Du(x,t), n(x) \rangle +\e \langle Dv(x), n(x) \rangle =0 \mbox{  on  } \p \Omega.
 \end{align}
 and
 $u^\e$ is a viscosity supersolution of
  \begin{align}
  \frac{\p u}{\p t}=a^{ij}(Du-\e Dv,t) D_iD_j u+b(Du-\e Dv,t)-\e \operatorname{tr}(a^{ij}(Du-\e Dv,t)) \mbox{  in   } \Omega,  \\
  \langle Du(x,t), n(x) \rangle -\e \langle Dv(x), n(x) \rangle =0 \mbox{  on  } \p \Omega.
  \end{align}

 We complete the proof by considering the following approximation of the function $Z$:
 $$Z_\e (x,y,t)=u_\e(y,t) -u^\e(x,t) -2 \phi\left(\frac{|y-x|}{2},t\right).$$
Then $Z_\e$ has a local max at $(x_\e, y_\e , t_\e)$
with  $(x_\e, y_\e , t_\e) \to (x_0,y_0,t_0)$ and $s_\e =|y_\e-x_\e|/2 \to s_0$.
As usual, choose $e_n=\frac{y_\e-x_\e}{|y_\e-x_\e|}$ and the maximum principle for semicontinuous functions \cite[Theorem 8.3]{CIL}  gives $b_{1,\e }, b_{2,\e } \in \R$ and $X_\e,Y_\e\in S^{n\times n}$
for any $\lambda >0$,
   \begin{equation*}
    (b_{1,\e }, \phi'(s_\e, t_\e)e_n, X_\e) \in \overline{\mathcal{P}}^{2,+}_{\ol{\Omega}\times (0,T)} u_\e(y_\e,t_\e),
   \end{equation*}
  \begin{equation*}
  (-b_{2,\e }, \phi'(s_\e, t_\e)e_n, Y_\e) \in \overline{\mathcal{P}}^{2,-}_{\ol{\Omega}\times (0,T)} u^\e(x_\e,t_\e),
  \end{equation*}
  \begin{equation*}
  b_{1,\e }+b_{2,\e }=2\phi_t(s_\e, t_\e),
  \end{equation*}

   \begin{equation*}
   -\left(\lambda^{-1}+\left\|M\right\| \right)I \leq
     \begin{pmatrix}
    X_\e & 0 \\
    0 & -Y_\e
    \end{pmatrix}
    \leq M+\lambda M^2,
   \end{equation*}
where $M=2D^2 \phi(s_\e,t_\e)$.
By the definition of viscosity solution for boundary problem, we have
If $y_\e \in \Omega$, then at $(x_\e ,y_\e ,t_\e )$,
\begin{equation}\label{y interior}
b_{1,\e} \leq tr (A(\phi^{\prime} e_n +\e Dv) X_\e) -b(\phi^{\prime} e_n+\e Dv)+\e \operatorname{tr}(A(Du+\e Dv)),
\end{equation}
and if $y_\e \in \p \Omega$, then at $(x_\e ,y_\e ,t_\e )$
\begin{eqnarray*}
\min & \left\{b_{1,\e }- tr (A(\phi^{\prime} e_n+\e Dv) X_\e) -b(\phi^{\prime} e_n+\e Dv)-\e \operatorname{tr}(A(Du+\e Dv)),\right. \\
 & \left.  \phi^{\prime} \langle e_n, n(y_\e)\rangle +\e \langle Dv(y_\e), n(y_\e) \rangle \right\} \leq 0,
\end{eqnarray*}
However, since $\Omega$ is convex and $\phi'\geq 0$,
$\phi^{\prime} \langle e_n, n(y_\e)\rangle +\e \langle Dv(y_\e), n(y_\e) \rangle \geq \e \delta .$
Thus equation \eqref{y interior} is valid no matter $y_\e$ lies in $\Omega$ or on $\p \Omega$.  Similarly,
If $x_\e \in \Omega$, then at $(x_\e ,y_\e ,t_\e )$
 \begin{equation}\label{x interior}
 -b_{2,\e} \geq  tr (A(\phi^{\prime} e_n-\e Dv) Y_\e) -b(\phi^{\prime} e_n -\e Dv)-\e \operatorname{tr}(A(\phi' e_n-\e Dv)),
 \end{equation}
and if $x_\e\in \p \Omega$, then at $(x_\e ,y_\e ,t_\e )$
 \begin{eqnarray*}
 \max & \{-b_{2,\e} - tr (A(\phi^{\prime} e_n-\e Dv ) Y_\e) -b(\phi^{\prime} e_n)+\e \operatorname{tr}(A(\phi'e_n-\e Dv)), \\
 & \phi' \langle e_n, n(x_\e)\rangle-\e \langle Dv(x_\e), n(x_\e) \rangle \} \geq 0.
 \end{eqnarray*}
 Observe that $\phi' \langle e_n, n(x_\e)\rangle-\e \langle Dv(x_\e), n(x_\e) \rangle \leq -\e \delta <0$,
 because $\Omega$ is convex and $\phi'\geq 0$.
 Therefore, equation \eqref{x interior}  is valid no matter $x_\e$ lies in $\Omega$ or on $\p \Omega$.
By passing to subsequences if necessary, we have
$b_{1,\e} \to b_1$, $b_{2,\e} \to b_2$, $X_\e  \to X$ and $Y_\e  \to Y$ as $\e  \to 0$. The limits satisfies
$$b_1+b_2=2\phi_t(s_0,t_0),$$
\begin{equation*}
   -\left(\lambda^{-1}+\left\|M\right\| \right)I \leq
     \begin{pmatrix}
    X & 0 \\
    0 & -Y
    \end{pmatrix}
    \leq M+\lambda M^2,
   \end{equation*}
Letting $\e \to 0$ in \eqref{y interior} and \eqref{x interior}, we obtain at $(y_0, t_0)$ $(x_0, t_0)$
$$b_1 \leq tr (A(\phi^{\prime} e_n ) X) -b(\phi^{\prime} e_n),$$
$$-b_2 \geq  tr (A(\phi^{\prime} e_n) Y) -b(\phi^{\prime} e_n),$$
The rest of the proof is the same as the proof of Theorem 1.1 in \cite{me1}.
\end{proof}

\begin{remark}
The same modulus of continuity estimate holds on manifolds for Neumann boundary problem.
The argument is the same as in Section 3.
\end{remark}

\subsection{Dirichlet Problem}
We consider the following quasilinear evolution equations:
\begin{align*}
\frac{\p u}{\p t}=a^{ij}(Du,t) D_iD_j u \mbox{  in   } \Omega\times (0,T),  \\
u(x,t)=0 \mbox{  on  } \p \Omega\times (0,T).
\end{align*}
Where $A(p,t)=\left(a^{ij}(p, t)\right)$ is positive semi-definite and $n(x)$ is the exterior unit normal vector at $x$.
As in \cite{AC2}, we assume that there exists a continuous function
$\alpha: \R_{+}\times [0, T] \to \R$ with
\begin{equation*}
0 < \alpha(R, t) \leq R^2 \inf_{|p|=R, (v \cdot p) \neq 0} \frac{v^{T}A(p, t)v}{(v \cdot p)^2},
\end{equation*}
For Dirichlet problem, we cannot formulate the same theorem as for Neumann problem.
B. Andrews' proof for regular solutions assumes concavity of the modulus of continuity
to rule out the case that maximum can occur on the boundary. Instead, we prove the following theorem,
which is a generalization of Theorem 4.2 in \cite{AC2} to viscosity solutions.

\begin{thm}
Let $\Omega \subset \RR^n$ be a smooth bounded and convex domain. Let $u$ be a continuous viscosity solution of the Dirichlet problem. Let $\vp_0$ be a modulus of continuity of $u(\cdot,0)$. Suppose $\vp$ is increasing and concave in the first variable and satisfies
$$\vp_t \geq  \alpha(|\vp^{\prime}|, t) \vp^{\prime \prime},$$
and $\vp(z,t) \geq \vp_0(z)$ for all $z\geq 0$.  Then $\vp(s,t)$ is a modulus of continuity for $u(s,t)$ for all $t>0$.
\end{thm}
\begin{proof}
The proof is exactly as the proof of Theorem 4.2 in \cite{AC2}
except one replaces the usual comparison principle by the comparison principle for viscosity solutions .
\end{proof}

\section{Acknowledgement}
The second author's research is partially supported
by the Priority Academic Program Development of Jiangsu Higher Education Institutions (PAPD).
Both authors would like to thank Professor Lei Ni for his interest.

\bibliographystyle{alpha}

\bibliography{myref}

\end{document}